\documentclass[12pt, reqno, draft]{amsart}

\usepackage{amssymb, amsmath, amsthm, url}
\usepackage{amscd}   
\usepackage{fullpage}

\usepackage{color}

\newtheorem{lemma}{Lemma}[section]

\newtheorem{theorem}[lemma]{Theorem}

\newtheorem{prop}[lemma]{Proposition}
\newtheorem{proposition}[lemma]{Proposition}
\newtheorem{cor}[lemma]{Corollary}
\newtheorem{corollary}[lemma]{Corollary}

\newtheorem{conjecture}[lemma]{Conjecture}

\newtheorem{claim*}{Claim}

\theoremstyle{definition}

\newtheorem{definition}[lemma]{Definition}

\newtheorem{remark}[lemma]{Remark}
\newtheorem{example}[lemma]{Example}

\let\a\alpha
\let\b\beta
\let\g\gamma

\let\f\phi
\let\l\lambda

\let\s\sigma
\let\t\tau
\let\G\Gamma

\newcommand{\bfe}{{\mathbf{e}}}
\newcommand{\bfF}{{\boldsymbol{F}}}
\newcommand{\bfk}{{\mathbf{k}}}
\newcommand{\bfm}{{\mathbf{m}}}
\newcommand{\bfx}{{\mathbf{x}}}

\newcommand{\bfzero}{\boldsymbol{0}}

\newcommand{\GG}{{\mathbb G}}

\newcommand{\PP}{{\mathbb P}}

\newcommand{\CC}{{\mathbb C}}

\newcommand{\ZZ}{{\mathbb Z}}

\newcommand{\calA}{{\mathcal A}}

\newcommand{\calI}{{\mathcal I}}
\newcommand{\calJ}{{\mathcal J}}

\newcommand{\calL}{{\mathcal L}}
\newcommand{\calM}{{\mathcal M}}

\newcommand{\calO}{{\mathcal O}}
\newcommand{\calP}{{\mathcal P}}

\newcommand{\calR}{{\mathcal R}}
\newcommand{\calS}{{\mathcal S}}
\newcommand{\calT}{{\mathcal T}}

\newcommand{\calV}{{\mathcal V}}

\newcommand{\OO}{{\mathcal O}}

\DeclareMathOperator{\sgn}{sgn}

\DeclareMathOperator{\rank}{rank}

\DeclareMathOperator{\id}{id}
\newcommand{\Relations}{\mathcal{R}}

\newcommand{\Linearspaces}{\mathcal{L}}
\newcommand{\Orbit}{\mathcal{O}}
\DeclareMathOperator{\superrank}{super-rank}

\numberwithin{equation}{section}
\numberwithin{table}{section}

\begin{document}


\title[Uniform bounds for dynamical Mordell--Lang]{%
On a uniform bound for the number of  exceptional linear subvarieties in the
dynamical Mordell--Lang conjecture
}
\date{\today}
\author{Joseph H. Silverman}
\address{Mathematics Department, Box 1917, Brown University, 
Providence, RI 02912 USA} 
\email{jhs@math.brown.edu} 
\author{Bianca Viray}
\address{Mathematics Department, Box 1917, Brown University, 
Providence, RI 02912 USA} 
\email{bviray@math.brown.edu} 
\subjclass[2010]{Primary: 37P15; Secondary: 11D45, 37P05}
\keywords{arithmetic dynamics, Mordell--Lang conjecture}

\thanks{The first author's research supported by NSF DMS-0854755.
The second author's research supported by NSF DMS-1002933.}

\begin{abstract}
Let $\f:\PP^n\to\PP^n$ be a morphism of degree~$d\ge2$ defined
over~$\CC$.  The dynamical Mordell--Lang conjecture says that the
intersection of an orbit $\Orbit_\f(P)$ and a subvariety
$X\subset\PP^n$ is usually finite. We consider the number of linear
subvarieties~$L\subset\PP^n$ such that the intersection
$\Orbit_\f(P)\cap L$ is ``larger than expected.'' When~$\f$ is the
$d^{\text{th}}$-power map and the coordinates of~$P$ are
multiplicatively independent, we prove that there are only finitely
many linear subvarieties that are ``super-spanned'' by~$\Orbit_\f(P)$,
and further that the number of such subvarieties is bounded by a
function of~$n$, independent of the point~$P$ and the degree~$d$.  More
generally, we show that there exists a finite subset $S$, whose
cardinality is bounded in terms of $n$, such that any $n+1$
points in $\Orbit_{\f}(P)\smallsetminus S$ are in linear general position
in $\PP^n$.
\end{abstract}


\maketitle

\section{The Dynamical Mordell--Lang Conjecture} 

The classical Mordell conjecture says that a curve~$C$ of genus
$g\ge2$ defined over a number field~$K$ has only finitely many
$K$-rational points.  One may view~$C$ as embedded in its
Jacobian~$J$, and then Mordell's conjecture may be reformulated as
saying that~$C$ intersects the finitely generated group~$J(K)$ in only
finitely many points.  Taking this viewpoint, Lang conjectured that
if~$\G\subset A$ is a finitely generated subgroup of an abelian
variety~$A$ and if $X\subset A$ is a subvariety of~$A$, then~$X\cap\G$ is
contained in a finite union of translates of proper abelian
subvarieties of~$A$. The Mordell--Lang conjecture for abelian
varieties was proven by Faltings~\cite{MR1109353,MR1307396}, building
on ideas pioneered by Vojta~\cite{MR1109352} in his alternative proof
of the original Mordell conjecture.
\par
The classical Mordell--Lang may be reformulated in dynamical terms as
follows. Let $P_1,\ldots,P_r$ be generators of~$\G$, and for
each~$1\le i\le r$, let $T_i:A\to A$ be the translation-by-$P_i$ map,
i.e., $T_i(Q)=Q+P_i$. Further let~$\calT$ be the group of self-maps
of~$A$ generated by~$T_1,\ldots,T_r$. Then~$\G$ is simply the complete
orbit of~$0$ by the group of maps~$\calT$, so the Mordell--Lang
conjecture is a statement about the intersection of an orbit and a
subvariety.
\par
The following is a dynamical analogue of the Mordell--Lang conjecture
for self-morphisms of algebraic varieties;
see~\cite{MR1259107,arxiv0805.1560}.

\begin{conjecture}[Dynamical Mordell--Lang Conjecture]
Let $\f:V\to V$ be a self-morphism of an algebraic variety defined
over~$\CC$, let $X\subset V$ be a subvariety, and let $P\in V(\CC)$.
Then
\[
  \bigl\{n\ge0 : \f^n(P)\in X\bigr\}
\]
is a finite union of arithmetic progressions \textup(where a single integer
is viewed as an arithmetic progression with common difference~$0$\textup).
\end{conjecture}

There are currently only a few scattered results in the literature
related to the Mordell--Lang conjecture in the dynamical
setting. These include results on \'etale maps, an analogue for
Drinfeld modules, and results for maps of various special types, for
example diagonal maps
$\f(z_1,\ldots,z_n)=\bigl(f_1(z_1),\ldots,f_n(z_n)\bigr)$;
see~\cite{MR2766180,arXiv:0712.2344,arxiv0805.1560,arXiv:0704.1333,arxiv0705.1954}.
\par
Write 
\[
  \calO_\f(P)=\{\f^n(P):n\ge0\}
\]
for the forward orbit of~$P$.  The intersection $X\cap\calO_\f(P)$ may be
infinite if there is some positive-dimensional subvariety~$Y\subset X$
that is periodic, i.e., $\f^N(Y)=Y$, since then $\f^k(P)\in Y$ implies
that $\f^{k+iN}(P)\in Y$ for all $i\ge0$.  If there is no such
subvariety, then one generally expects the intersection
$X\cap\calO_{\f(P)}$ to be finite.
\par
We now turn the Mordell--Lang problem around and consider the set of
subvarieties~$X$ whose intersection with~$\Orbit_\f(P)$ is finite, but
``larger than one would expect.'' In this paper we restrict attention
to self-maps of~$\PP^n$ and linear subspaces~$X$, which still present
sufficiently many difficulties to merit study. For example,
if~$\f:\PP^2\to\PP^2$, then it seems plausible that there should be
only finitely many lines in~$\PP^2$ that contain three (or more)
points of the orbit~$\Orbit_\f(P)$.  More generally, one might expect
that there are only finitely many hyperplanes~$H$ in~$\PP^n$ that
contain $n+2$ points of the orbit~$\Orbit_\f(P)$, but this is not
quite true. The problem is that there might be some lower dimensional
linear space~$L\subset\PP^n$ that contains $n+1$ points
of~$\Orbit_\f(P)$, and then every hyperplane~$H$ containing~$L$ and
having non-empty intersection with $\Orbit_\f(P)\smallsetminus L$ will
\emph{a fortiori} contain $n+2$ points of~$\Orbit_\f(P)$. The solution
is to look only at hyperplanes~$H$ containing $n+2$
points~$Q_1,\ldots,Q_{n+2}$ of~$\Orbit_\f(P)$ such that~$H$ is spanned
by every subset of~$\{Q_1,\ldots,Q_{n+2}\}$ consisting of $n+1$
points.
\par
We will say that an $(r-1)$-dimensional linear space $L\subset\PP^n$
is \emph{super-spanned} by the set of points $\{Q_0,\ldots,Q_r\}$ if every 
subset consisting of~$r$ points spans~$L$. With this definition, we can state
our main conjecture.

\begin{conjecture}
\label{conjecture:finmanylinsp}
Let $\f:\PP^n\to\PP^n$ be a morphism of degree~$d\ge2$
defined over~$\CC$, and
let~$P\in\PP^n(\CC)$ be a point whose orbit~$\calO_\f(P)$ is Zariski
dense in~$\PP^n$. Let $r\ge1$. Then there are only finitely many
linear subspaces $L\subset\PP^n$ of dimension $r-1$ such that
$L\cap\calO_\f(P)$ contains a set of~$r+1$ points that
super-spans~$L$. Furthermore, the number of such linear subspaces can
be bounded by a function that depends only on $n$ and $d$.
\end{conjecture}

Conjecture~\ref{conjecture:finmanylinsp} may be viewed as saying that
the orbit $\calO_{\f(P)}$ is ``almost'' in linear general position in
the following sense: after removing finitely many linear subspaces
of~$\PP^n$, no $r + 2$ points of~$\calO_{\f(P)}$ are contained in any
of the remaining $(r-1)$-dimensional linear spaces, and moreover, that
the number of linear subspaces that need to be removed is bounded by a
function that depends only on $n$ and $d$.  If, in addition, one knows
that $L\cap\calO_\f(P)$ is finite for all linear spaces $L$, then one
gets a stronger version of ``almost'' in linear general position,
i.e., there exists a finite subset $S\subseteq \Orbit_{\f}(P)$, whose
cardinality is bounded in terms of $n$ and $d$, such that any $n+1$
points in $\Orbit_{\f}(P)\smallsetminus S$ are in linear general
position in $\PP^n$.

In full generality, Conjecture~\ref{conjecture:finmanylinsp} seems
difficult. The primary result in this paper is a proof of the
conjecture for the $d^{\text{th}}$-power map, under the (possibly)
weaker assumption that the coordinates of the point~$P$ are
multiplicatively independent. In this case, we are able to prove a
uniform bound that is independent of both~$P$ \emph{and}~$d$.

\begin{theorem}
\label{theorem:mainthm}
Let $d\ge2$, let
\[
  \f\bigl([z_0,\ldots,z_n]\bigr)=[z_0^d,\ldots,z_n^d]
\]
be the $d^{\text{th}}$-power map on~$\PP^n$, and let
$P\in\PP^n(\CC)$ be a point whose coordinates are nonzero and
multiplicatively independent.  Then
Conjecture~$\ref{conjecture:finmanylinsp}$ is true for~$\f$ and~$P$.
More precisely, the number of super-spanned linear subspaces~$L$ is
bounded solely in terms of~$n$, independent of the point~$P$ and the
degree~$d$.
\end{theorem}

The proof of Theorem~\ref{theorem:mainthm} can be adapted to prove the
following uniform bound for the number of points in~$\Orbit_\f(P)\cap
L$.  We give the proof of Corollary~\ref{cor:maincor} in
Section~\ref{section:sizeintersection}.

\begin{cor}
\label{cor:maincor}
Let $\f$ and $P$ be as in the statement of
Theorem~$\ref{theorem:mainthm}$.  Then for any linear subspace $L$,
the intersection \text{$\OO_{\f}(P)\cap L$} is finite, and its size is
bounded solely in terms of~$n$, independent of $P$ and $d$.
\end{cor}

We conclude this introduction by giving a brief overview of the key
steps in the proof of Theorem~\ref{theorem:mainthm}.  We
consider~$r+1$ arbitrary iterates of~$\f$ applied to~$P$, say
\[
  Q_0=\f^{m_0}(P),\quad Q_1=\f^{m_1}(P),\;\ldots\quad Q_r=\f^{m_r}(P),
\]
and we assume that~$Q_0,\ldots,Q_r$ super-span a linear subspace~$L$ of
dimension~$r$. The fact that these points lie in~$L$
means that the $(r+1)$-by-$(n+1)$ matrix~$A$ whose rows are the
points~$Q_0,\ldots,Q_r$ has rank~$r$. Hence all of its
$(r+1)$-by-$(r+1)$ minors have zero determinant, and expanding these
determinants as sums over the permutation group~$\calS_{r+1}$ gives various
linear combinations of products of powers of the coordinates of the
point~$P$. We then apply a deep theorem on uniform bounds for the
number of non-degenerate solutions to $S$-unit equations
$u_1+\cdots+u_N=0$ due to Evertse, Schlickewei, and
Schmidt~\cite{MR1923966}.  If none of the subsums over subsets
of~$\calS_{r+1}$ vanishes, the proof is essentially complete. (This is
where we use the multiplicative independence of the coordinates
to~$P$, since what we really get is that certain products of powers of
the coordinates take on only finitely values.) However, it is
certainly possible for subsums of the determinant sums to vanish, so
we need to consider all possible partitions of~$\calS_{r+1}$
associated to vanishing subsums. This leads to a fairly elaborate
argument in which we prove the desired result for ``good'' partitions,
while also characterizing the ``bad'' partitions where we don't get
finiteness and showing that these bad partitions contradict the
assumption that the matrix~$A$ has super-rank~$r$.

\section{Related Work on Bounding the Number of Exceptional Subvarieties in
Diophantine Problems} 
In this section we briefly mention some earlier work in which authors
have used $S$-unit equation bounds to show that most Diophantine
problems of various types have at most the generically expected
number of solutions.
\par
We start with work of Evertse, Gy{\H{o}}ry, Stewart, and
Tijdeman~\cite{MR939471} in which they show that up to equivalence,
there are only finitely many equations of the form $ax+by=c$ that have
three or more solutions in $S$-units $x,y\in R_S^*$.  We sketch their
proof. Assuming that there are three solutions, one can eliminate~$a/c$
and~$b/c$ to obtain the determinantal equation
\[
  0 = \det\begin{pmatrix} x_1&y_1&1\\x_2&y_2&1\\ x_3&y_3&1\\\end{pmatrix}
  = x_1y_2-x_2y_1-x_1y_3+x_3y_1+x_2y_3-x_3y_2.
\]
This gives a six-term $S$-unit equation. If no subsum vanishes,
they're essentially done. To finish the proof,
they do a case-by-case analysis
of five special cases where various subsums vanish.  This proof has
some features in common with our proof of
Theorem~\ref{theorem:mainthm}, but our determinants are of arbitrary
size, so the ``case-by-case'' analysis must cover all possible ways in
which subsums of a multi-term sum can vanish.
\par
In the higher dimensional case, Evertse~\cite{MR2093164} considers a
finitely generated subgroup~$\G$ of~$K^n$. He proves that there is
a finite union~$\calA$ of $\G$-equivalence classes of $n$-tuples
$(a_1,\ldots,a_n)$ such that for all $n$-tuples not in~$\calA$, the
set of non-degenerate solutions to the equation
\[
  a_1x_1+\cdots+a_nx_n=1,
  \quad x_1,\ldots,x_n\in\G,
\]
lies in the union of~$2^n$ proper linear subspaces of~$K^n$.  His
proof relies on a result of Laurent~\cite{MR767195} for the number of
non-degenerate points in~$X\cap\G$, where~$X$ is an algebraic
subvariety of~$\GG_m^n$.  Remond~\cite{MR2171293} has generalized
Evertse's result to semi-abelian varieties.
\par
Perhaps closest to the present work is a paper of Schlickewei and
Viola~\cite{MR1785411}.  For fixed $\a_1,\ldots,\a_n\in K^*$ such that
no ratio $\a_i/\a_j$ is a root of unity, they consider solutions to
the determinantal equation
\begin{equation}
  \label{eqn:aiyjmatrix}
  F(y_2,\ldots,y_n) = 
  \det \begin{pmatrix}
    1 & 1 & \cdots & 1 \\
    \a_1^{y_2} & \a_2^{y_2} & \cdots & \a_n^{y_2} \\
    \vdots  & & \cdots & \vdots \\
    \a_1^{y_n} & \a_2^{y_n} & \cdots & \a_n^{y_n} \\
  \end{pmatrix} = 0,
  \qquad y_2,\ldots,y_n\in\ZZ.
\end{equation}
They prove that the equation~\eqref{eqn:aiyjmatrix} has at most
$\exp\bigl((6n!)^{3n!}\bigr))$ solutions with the property that
\emph{all} proper subdeterminants of the matrix appearing
in~\eqref{eqn:aiyjmatrix} are non-zero. They conjecture that a similar
result is true under the weaker assumption that every $(n-1)$-by-$n$
and every $n$-by-$(n-1)$ submatrix has rank~$n-1$.
\par
The theorem and conjecture of Schlickewei and Viola are both stronger
and weaker than our main results. Our results are weaker in two
ways. First, we assume that~$\a_1,\dots,\a_n$ are multiplicatively
independent. Second, we essentially end up considering equations of
the form
\[
  F(d^{z_2},\ldots,d^{z_n}) = 0,
\]
where $d\ge2$ is a fixed integer, so we require that Schlickewei and Viola's 
variables~$y_2,\ldots,y_n$ be powers of~$d$.
\par
Our results are also stronger in two ways. First, we only require that
every $(n-1)$-by-$n$ submatrix of~\eqref{eqn:aiyjmatrix} have
rank~$n-1$, which is less stringent than even their conjectural requirement.
Second, we work much more generally with an $r$-by-$n$ matrix, with the
assumption that it has rank~$r-1$ and that each of its
$(r-1)$-by-$n$ submatrices also has rank~$r-1$. It would be
interesting to see if Schlickewei and Viola's result is true in this
$r$-by-$n$ setting, subject of course to their strong assumption that
every subdeterminant is nonzero.
\par
Both our proof and the proof of Schlickewei and Viola use in a
fundamental way the theorem of Evertse, Schlickewei, and
Schmidt~\cite{MR1923966} bounding the number of solutions to linear
equations taking values in a finitely generated group. The proofs
resemble one another in that they require intricate manipulations of
the various ways in which subsums of a determinantal sum can vanish,
but the proofs differ in many details due to their differing assumptions.

\section{Exceptional Linear Subspaces} 

For the remainder of this paper we fix an algebraically closed
field~$K$ of characteristic~$0$.  Unless we indicate otherwise, all
varieties, maps, and points are assumed to be defined over~$K$.  We
also note that we use square brackets to denote homogeneous
coordinates.  Throughout $r$ denotes an integer in $\{1, \ldots, n\}$.

\begin{definition}
Let $L\subsetneq\PP^n$ be a linear space of dimension~$r-1$,
so~$r$ points in general position on~$L$ will span~$L$. 
A set of~\text{$r+1$} points
\[
  S = \{Q_0,Q_1,\ldots,Q_r\}\subset L
\]
is said to \emph{super-span~$L$} if every subset of~$S$
containing~$r$ points spans~$L$.  
\par
Similarly, an~\text{$(r+1)$}-by-\text{$(n+1)$} matrix~$A$ is said to
have \emph{super-rank~$r$} if~$A$ has rank~$r$ and further every
submatrix consisting of~$r$ rows of~$A$ has rank~$r$.  From these
definitions, if we let~$S$ denote the set of points in~$\PP^n$
corresponding to the rows of~$A$, then~$A$ has super-rank~$r$ if and
only if~$S$ super-spans an~\text{$(r-1)$}-dimensional linear subspace
of~$\PP^n$.
\end{definition}

\begin{definition}
Let $\f:\PP^n\to\PP^n$ be a morphism, and let $P\in\PP^n$ be a point.
For $r\ge1$ we define the \emph{set of exceptional linear spaces
  for~$\f$ and~$P$} to be the set
\[
  \Linearspaces_{\f,P}^r
    = \left\{ L\subset\PP^n : 
       \begin{tabular}{l}
          $L$ is a linear space of dimension $r-1$ and\\
          $L\cap\Orbit_\f(P)$ contains points $Q_0,\ldots,Q_r$\\
          such that $\{Q_0,\ldots,Q_r\}$ super-spans $L$\\
       \end{tabular}
     \right\}.
\]
\end{definition}

Using this notation, we can rewrite
Conjecture~\ref{conjecture:finmanylinsp} as follows.

\begin{conjecture}
\label{conjecture:linearspacefiniteness}
Let $\f:\PP^n\to\PP^n$ be a morphism of degree~$d\ge2$, and
let~$P\in\PP^n$ be a point whose orbit~$\Orbit_\f(P)$ is Zariski dense
in~$\PP^n$. Then for all $r\ge1$, the set of exceptional linear
spaces~$\Linearspaces_{\f,P}^r$ is finite, and
$\#\Linearspaces_{\f,P}^r$ may be bounded solely in terms of~$n$
and~$d$.
\end{conjecture}

\begin{remark}
\label{remark:someexcept}
For a given map, it is easy to find initial points that lead to at
least a few exceptional linear spaces. Thus fix a
morphism~$\f:\PP^n\to\PP^n$, and let $0<m_0<m_1<\cdots<m_r$ be a list
of integers. Treating the coordinates of $P=[\a_0,\ldots,\a_n]$ as
indeterminates, the condition that
$\f^{m_0}(P),\f^{m_1}(P),\ldots,\f^{m_r}(P)$ span a linear space of
dimension~$r-1$ is equivalent to requiring that $n-r+1$ determinants
vanish, so it puts $n-r+1$ constraints on the coordinates of~$P$. The
super-spanning condition is Zariski open, so we'll ignore it. If we
choose another list of $r+1$ iterates to lie in a linear space of
dimension~$r-1$, we get another $n-r+1$ conditions on~$P$. Hence
generically it should be possible to choose~$P$ so that
\[
  \#\calL_{\f,P}^r \ge \left\lfloor\frac{n}{n-r+1}\right\rfloor.
\]
(Note that the coordinates of~$P$ and the linearity conditions are
homogeneous.)  So for example, for most~$\f:\PP^n\to\PP^n$ one expects
that there exist initial points~$P$ such that the orbit of~$P$
super-spans~$n$ distinct hyperplanes, which is the case $r=n$.  This
also suggests that any effective bound for
\[
  \max\left\{M : 
     \begin{tabular}{@{}l@{}}
       there exist $0<m_0<\cdots<m_r=M$ such\\
       that $\f^{m_0}(P),\ldots,\f^{m_r}(P)$ superspan\\
       a linear subspace of dimension $r-1$\\
     \end{tabular}
  \right\}
\]
must depend on~$P$, since in our construction
we can take $m_r$ to be arbitrarily large.
\end{remark}

\begin{example}
We illustrate Remark~\ref{remark:someexcept} using the map
\[
  \f:\PP^2\longrightarrow\PP^2,\qquad
  \f\bigl([x,y,z]\bigr)=[x^2,y^2,z^2].
\]
We will find an initial point $P=[\a,\b,\g]$ so that
$\#\calL_{\f,P}^1\ge2$, i.e., so that~$\Orbit_\f(P)$ super-spans at
least two lines.  The condition that~$P$,~$\f(P)$, and~$\f^2(P)$ be
colinear, i.e., they super-span a line, is
\begin{equation}
  \label{eqn:det012}
  \det\begin{pmatrix}
    \a & \b & \g \\
    \a^2 & \b^2 & \g^2 \\
    \a^4 & \b^4 & \g^4 \\
  \end{pmatrix}
  = \a\b\g(\a-\b)(\b-\g)(\g-\a)(\a+\b+\g)
  = 0.
\end{equation}
Similarly, the condition that $P$,~$\f^3(P)$, and~$\f^4(P)$ are colinear is
\begin{equation}
  \label{eqn:det034}
  \det\begin{pmatrix}
    \a & \b & \g \\
    \a^8 & \b^8 & \g^8 \\
    \a^{16} & \b^{16} & \g^{16} \\
  \end{pmatrix}
  = \a\b\g(\a-\b)(\b-\g)(\g-\a)h(\a,\b,\g)
  = 0,
\end{equation}
where~$h$ is a complicated homogeneous polynomial of degree~$19$.
Hence $\f$ and~$P$ will have two exceptional lines in~$\PP^2$, i.e.,
$\#\calL_{\f,P}^2\ge2$, if~$P$ is chosen to satisfy the two
simultaneous equations~\eqref{eqn:det012} and~\eqref{eqn:det034} with
$\a\b\g\ne0$ and~$\a,\b,\g$ distinct. Dehomogenizing $\g=1$ and
discarding solutions in which~$\a$ or~$\b$ is a root of unity, we find
that~$P$ has the form~$P=[\a,\b,1]$ with~$\a$ and~$\b$ roots of the
polynomial
\[
  2x^6+6x^5+5x^4+5x^2+6x+2,
\]
such that $\a + \b = - 1$.  The roots of this polynomial have the form
$\{x_1,\bar x_1,x_1^{-1},\bar x_1^{-1},x_2,\bar x_2\}$, with $x_1 +
x_2 = x_1^{-1} + \bar x_1^{-1} = -1$. Since we want~$\a$ and~$\b$ to
be multiplicatively independent, it suffices to
take~$\a=x_1\approx-1.6243 - 0.7812i$ and
$\b=x_2\approx0.6243+0.7812i$. Since $|\a|>1$ and $|\b|=1$, it is
clear that they are multiplicatively independent.
\end{example}

\section{Finiteness of Exceptional Linear Subspaces for the $d$-Power Map} %

Our main result says that subject to a multiplicative independence
assumption on the coordinates of the point~$P$, a strengthened form of
the finiteness conjecture
(Conjecture~\ref{conjecture:linearspacefiniteness}) is true for the
$d^{\text{th}}$-power map on~$\PP^n$.

\begin{definition}
Let $P = [\alpha_0,\alpha_1,\cdots,\alpha_n] \in \PP^n$ be a point
not contained in any coordinate hyperplane, i.e., satisfying
$\a_0\a_1\cdots\a_n\ne0$.  We define the \emph{multiplicative relation
  set of~$P$} to be the set
\[
  \Relations(P) = \left\{ \bfe=(e_0,\ldots,e_n)\in\ZZ^{n+1} :
  \sum_{i=0}^n e_i=0 \quad\text{and}\quad \prod_{i=0}^n \a_i^{e_i}=1
  \right\}.
\]
We observe that~$\Relations(P)$ is a sublattice of~$\ZZ^{n+1}$.
\end{definition}

\begin{theorem}
\label{theorem:finmanyH}
Let~$\f:\PP^n\to\PP^n$ be the $d^{\text{th}}$-power map
\begin{equation}
  \label{eqn:dthpowermap}
  \f\bigl([z_0,\ldots,z_n]\bigr)
     =[z_0^d,\ldots,z_n^d].
\end{equation}
Fix a point $P \in \PP^n$ whose relation set satisfies
$\Relations(P)=\{\bfzero\}$. Then for all~$0\le r\le n$, the set of
exceptional linear spaces~$\Linearspaces^r_{\f,P}$ is finite. Further
there is an upper bound for~$\#\Linearspaces^r_{\f,P}$ that depends
only on~$n$, independent of the point~$P$ and the degree~$d$.
\end{theorem}

\begin{remark}
Theorem~\ref{theorem:finmanyH} deals with linear subspaces of
arbitrary dimension, but on first reading it may be easier for the
reader to consider the case of hyperplanes, i.e., $r=n$. In
particular, setting~$r=n$ means that the set~$\calP_{r,n}$ defined
after the proof of Lemma~\ref{lemma:MtoHsurjective} contains only one
element, namely the identity map on~$\{0,\ldots,n\}$, which
significantly simplifies the exposition.
\end{remark}

\begin{proof}
First, we observe that if any coordinate of~$P$ is~$0$, then we can
discard that coordinate and work on a lower dimensional projective
space.  So by induction on~$n$, we may assume that~$P$ does not lie in
any of the coordinate hyperplanes.
\par
Second, we note that the assumption that $\Relations(P)=\{\bfzero\}$
implies that~$P$ is not preperiodic, since $\f^m(P)=P$ implies that
$\a_i^{d^m-1}\a_j^{1-d^m}=1$ for all~$i$ and~$j$.  (Of course, if~$P$
is preperiodic, it is obvious that $\Linearspaces^r_{\f,P}$ is finite,
since a finite set of points (super)spans only finitely many linear
spaces. But it is not clear that there is a uniform bound for
$\#\Linearspaces^r_{\f,P}$ independent of $P$.)
\par
For the initial part of our proof, it suffices to assume that~$P$ is
not preperiodic and does not lie in a coordinate hyperplane.
For possible future applications, we start with only these assumptions,
and we will indicate where our proof first uses the stronger
condition $\Relations(P)=\{\bfzero\}$. 
\par
Write 
\[
  P = [\alpha_0 , \alpha_1 , \ldots , \alpha_n],
\]
where by assumption we have $\a_0\a_1\cdots\a_n\ne0$.  
For any $(r+1)$-tuple of integers $\bfm=(m_0,\ldots,m_r)$, we define an
$(r+1)$-by-$(n+1)$ matrix (depending on~$P=[\a_0,\ldots,\a_n]$)
\[
  A_\bfm = \left(\a_j^{d^{m_i}}\right)_{\substack{0\le i\le r\\ 0\le j\le n\\}}.
\]
In other words, the~$i^{\text{th}}$~row of~$A_\bfm$, considered as a
point in~$\PP^n$, is the point~$\f^{m_i}(P)$.  We then define a
collection of exceptional \text{$r+1$}-tuples of iterates by
\[
  \calM_P = \bigl\{\bfm\in\ZZ^{r+1} : 0\le m_0<\cdots<m_r
    \text{ and }  \superrank A_\bfm = r \bigr\}.
\]
\par
We also define
\[
  L_\bfm = \left(\begin{tabular}{@{}l@{}}
    the linear subspace of~$\PP^n$ spanned by the\\
    points whose homogeneous coordinates\\
    are the row vectors of the matrix~$A_\bfm$\\
  \end{tabular}\right)
\]
For~$\bfm\in\calM_P$, the rank condition on~$A_\bfm$ implies
that~$L_\bfm$ is a linear subspace of exact dimension~$r-1$, and the
super-rank condition says that~$L_\bfm$ is spanned by any~$r$ of
the~$r+1$ rows of~$A_\bfm$.

\begin{lemma}
\label{lemma:MtoHsurjective}
The map
\begin{equation}
  \label{eqn:MtoHypmtoHm}
  \calM_P \longrightarrow \Linearspaces^r_{\f,P},\qquad
  \bfm\longrightarrow L_\bfm,
\end{equation}
is surjective.
\end{lemma}
\begin{proof}
The fact that $L_\bfm$ lies in~$\Linearspaces^r_{\f,P}$ follows directly
from the definitions of~$\calM_P$ and~$\Linearspaces^r_{\f,P}$. For the
surjectivity, let~$L\in\Linearspaces^r_{\f,P}$. Then~$L$ contains~$r+1$
points in the orbit~$\Orbit_\f(P)$, which we can label as
$\f^{m_0}(P),\ldots,\f^{m_r}(P)$ with $0\leq m_0<m_1<\cdots<m_r$.  We set
$\bfm=(m_0,\ldots,m_r)$. Since
\[
  \f^{m_j}(P) = \bigl[\a_0^{d^{m_j}},\a_1^{d^{m_j}},\cdots,\a_n^{d^{m_j}}\bigr],
\]
the definition of~$\Linearspaces^r_{\f,P}$ implies that the
matrix~$A_\bfm$ has super-rank~$n$, so~$\bfm$ is in~$\calM_P$, and by
construction, the image of~$\bfm$ in~$\Linearspaces^r_{\f,P}$
is~$L$. 
\end{proof}

\begin{remark}
The map~\eqref{eqn:MtoHypmtoHm} in Lemma~\ref{lemma:MtoHsurjective}
need not be injective. For example, if there is a linear
space~$L\in\Linearspaces^r_{\f,P}$ such that $L\cap\Orbit_\f(P)$
contains~$r+2$ points such that every subset consisting of~$r$ points
spans~$L$, then the cardinality of the preimage of $L$ is at least
$\binom{r+2}{r+1} = r+2$.
\end{remark}

Returning to the proof of Theorem~\ref{theorem:finmanyH}, we define
some additional notation. First, in order to select~$r+1$ columns of
a matrix~$A_\bfm$, we look at the set of maps
\[
  \calP_{r,n} = \left\{
     \begin{tabular}{@{}l@{}}
        strictly increasing maps\\
        $p:\{0,\ldots,r\}\to\{0,\ldots,n\}$\\
     \end{tabular}
  \right\}.
\]
For example, if~$r=n$, then~$p$ is necessarily the identity map, while
if~$r=n-1$, then there are exactly~$n+1$ maps~$p$, each of which is
determined by the one value between~$0$ and~$n$ that is not in the
image of~$p$. For each~$p\in\calP_{r,n}$, we define an
\text{$(r+1)$}-by-\text{$(r+1)$} submatrix of~$A_\bfm$ by
\[
  A_{\bfm,p} = \left(\a_{p(j)}^{d^{m_i}}\right)_{\substack{0\le i\le r\\ 0\le j\le r\\}}.
\]
In other words, the $(r+1)$-by-$(r+1)$ matrix~$A_{\bfm,p}$ is obtained
from the $(r+1)$-by-$(n+1)$ matrix~$A_\bfm$ by taking
columns~$p(0),p(1),\ldots,p(r)$.
\par
To ease notation, we also let
\begin{align*}
  k_i(\bfm) &= d^{m_i}&&\text{for $0\le i\le r$,} \\
  u_{\s,p}(\bfm) 
    &= \alpha_{p(0)}^{k_{\sigma(0)}(\bfm)}\alpha_{p(1)}^{k_{\sigma(1)}(\bfm)}
             \cdots\alpha_{p(r)}^{k_{\sigma(r)}(\bfm)}
    &&\text{for $\s\in\calS_{r+1}$ and $p\in\calP_{r,n}$.}
\end{align*}
With this notation, the determinant of the matrix~$A_{\bfm,p}$ is
\[
  \det (A_{\bfm,p}) =
  \det\bigl( \alpha_{p(j)}^{k_i(\bfm)}\bigr)_{\substack{0\le i\le r\\ 0\le j\le r\\}}
     = \sum_{\sigma \in \calS_{r + 1}} \sgn(\sigma)u_{\sigma,p}(\bfm).
\]
\par 
To simplify notation, we will often suppress the dependence on~$\bfm$ and
just write~$k_i$ and~$u_{\s,p}$, but we stress that our eventual goal is
to show that there are only finitely many~$\bfm$ satisfying certain
conditions, so most of our formulas should be viewed as relations on
the unknown quantity~$\bfm$.

The following elementary lemma will be useful.

\begin{lemma}
\label{lemma:dadbeqdxdy}
Let $d\ge2$ be an integer, and let $a,b,x,y$ be non-negative integers.
Suppose that
\begin{equation}
  \label{eqn:dadbeqdxdy}
  d^a-d^b = d^x - d^y.
\end{equation}
Then $\{a,y\}=\{b,x\}$, i.e.,  either
\[
  \text{$(a=x$ and $b=y)$\quad or\quad $(a=b$ and $x=y)$.}
\]
\end{lemma}
\begin{proof}
Multiplying the relation~\eqref{eqn:dadbeqdxdy} by~$-1$ if necessary,
we may assume that $a\ge b$, and then since the left-hand side is
positive, we also have $x\ge y$. From~\eqref{eqn:dadbeqdxdy} it is
clear that if $a=b$, then $x=y$, and similarly, if~$x=y$, then~$a=b$.
We are thus reduced to the case that $a>b$ and
$x>y$. Factoring~\eqref{eqn:dadbeqdxdy} gives
\[
  d^b(d^{a-b}-1) = d^y(d^{x-y}-1).
\]
Since~$d^{a-b}-1$ and~$d^{x-y}-1$ are both relative prime to~$d$ (this
is where we use the assumption that $a>b$ and $x>y$), we find that
$d^b=d^y$, so $b=y$. It is then clear from~\eqref{eqn:dadbeqdxdy} 
that also $a=x$.
\end{proof}

We next use Lemma~\ref{lemma:dadbeqdxdy} to show that the values of the
products $u_{\s,p}(\bfm)$ determine the value of~$\bfm$

\begin{lemma}
\label{lemma:MtoPinjective}
Let $p\in\calP_{r,n}$ have the property that the point
\[
  [\a_{p(0)},\a_{p(1)},\ldots,\a_{p(r)}]
\]
is not preperiodic for the $d^{\text{th}}$-power map, i.e., at least
one of the ratios~$\a_{p(j)}/\a_{p(0)}$ is not a root of unity.  Then
the map
\begin{equation}
  \label{eqn:MtoPn11map}
  \calM_P\longrightarrow\PP^{(r+1)!- 1},\qquad
  \bfm \longrightarrow \bigl[u_{\s,p}(\bfm)\bigr]_{\s\in\calS_{r+1}},
\end{equation}
is injective.
\end{lemma}
\begin{proof}
Suppose that $\bfm$ and $\widetilde\bfm$ are elements of~$\calM_P$
that have the same image in~$\PP^{(r+1)!-1}$. This means that
\begin{equation}
  \label{eqn:prodaandtildea}
  \prod_i\alpha_{p(i)}^{{k_{\sigma(i)}}} \alpha_{p(i)}^{{\widetilde{k}_{\tau(i)}}} 
   = \prod_i\alpha_{p(i)}^{{k_{\tau(i)}}} \alpha_{p(i)}^{{\widetilde{k}_{\sigma(i)}}}, 
   \quad \textup{ for all $\sigma,\tau \in S_{r+1}$.}
\end{equation}
Applying~\eqref{eqn:prodaandtildea} with $\s=\id$ and $\t$ the
transposition $\t=(0j)$ gives the relation
\begin{equation}
  \label{eqn:aja0kjtildekj}
  (\alpha_{p(j)}/\alpha_{p(0)})^{{k_j} - {\widetilde{k}_j} - {k_0} +  {\widetilde{k}_0}}=1.
\end{equation}
This holds for all~$0\le j\le n$. We are assuming that at least one of
the ratios~$\a_{p(j)}/\a_{p(0)}$ is not a root of unity,
say~$\a_{p(t)}/\a_{p(0)}$ is not a root of unity. Then the exponent on
the left-hand side of~\eqref{eqn:aja0kjtildekj} must vanish, so we
find that
\[
  k_t - k_0 = \tilde{k}_t-\tilde{k}_0.
\]
Using the definition of~$k_j=k_j(\bfm)$, this is equivalent to the equation
\begin{equation}
  \label{dmjdm0dtildemj}
  d^{m_t} - d^{m_0} = d^{\widetilde{m}_t} - d^{\widetilde{m}_0}.
\end{equation}
The fact that~$\bfm\in\calM_P$ means
that~$m_t>m_0$, so~\eqref{dmjdm0dtildemj} and
Lemma~\ref{lemma:dadbeqdxdy} imply that
\begin{equation}
  \label{eqn:m0tildem0}
  m_0=\widetilde{m}_0\quad\text{and}\quad m_t=\widetilde{m}_t.
\end{equation}
\par
We now take an arbitrary index~$i$ and consider the multiplicative
relation~\eqref{eqn:prodaandtildea} for the permutations $\s=(0it)$
and $\t=(0ti)$.  After canceling common terms from the two sides of
the equation, we are left with the formula
\begin{equation}
  \label{eqn:a0kiaiktatk0}
  \a_{p(0)}^{k_i}\a_{p(i)}^{k_t}\a_{p(t)}^{k_0}
  \a_{p(0)}^{\tilde{k}_t}\a_{p(t)}^{\tilde{k}_i}\a_{p(i)}^{\tilde{k}_0}
  =
  \a_{p(0)}^{\tilde{k}_i}\a_{p(i)}^{\tilde{k}_t}\a_{p(t)}^{\tilde{k}_0}
  \a_{p(0)}^{k_t}\a_{p(t)}^{k_i}\a_{p(i)}^{k_0}.
\end{equation}
However, we already know from~\eqref{eqn:m0tildem0} that
$k_0=\tilde{k}_0$ and $k_t=\tilde{k}_t$. This allows us to cancel many
of the terms in~\eqref{eqn:a0kiaiktatk0}, and we find that
\[
  (\a_{p(0)}/\a_{p(t)})^{k_i-\tilde{k}_i}=1.
\]
We know that~$\a_{p(0)}/\a_{p(t)}$ is not a root of unity, since
that's how we chose~$t$, so it follows that $k_i=\tilde{k}_i$. This is
true for all~$i$, and since~$k_i=d^{m_i}$ and
$\tilde{k}_i=d^{\tilde{m}_i}$, we have proven that $m_i=\tilde{m}_i$
for all~$0\le i\le r$.  Hence~$\bfm=\tilde{\bfm}$, which completes the
proof that the map~\eqref{eqn:MtoPn11map} is injective.
\end{proof}

We resume the proof of Theorem~\ref{theorem:finmanyH}. Our goal is to
show that the set $\Linearspaces^r_{\f,P}$ is finite, so in view of
the surjectivity of the map~\eqref{eqn:MtoHypmtoHm} in
Lemma~\ref{lemma:MtoHsurjective}, it suffices to show that the
set~$\calM_P$ is finite.  Let~$\bfm\in\calM_P$. Then the
$(r+1)$-by-$(n+1)$ matrix~$A_\bfm$ has rank~$r$, so all of its
$(r+1)$-by-$(r+1)$ minors vanish. In our notation, 
\begin{equation}
  \label{eqn:detAmeq0}
  \det A_{\bfm,p} = \sum_{\s\in\calS_{r+1}} \sgn(\s)u_{\s,p}(\bfm) = 0
  \qquad\text{for all $p\in\calP_{r,n}$.}
\end{equation}
We will use the following deep result on~$S$-unit equations, which
we will apply with
\[
  \G = \{\text{subgroup of $K^*$ generated by $-1,\a_0,\ldots,\a_n$}\}.
\]

\begin{theorem}
\label{theorem:Suniteqn}
Let~$K$ be a field of characteristic~$0$, let~$\Gamma$ be a finitely
generated subgroup of~$K^*$, and let~$a_0,\ldots,a_N\in\Gamma$. Then
the equation
\[
  a_0u_0+a_1u_1+\cdots+a_Nu_N=0
\]
has only finitely many solutions $[u_0,\ldots,u_N]\in\PP^N(K)$ satisfying
\[
  u_0,\ldots,u_N\in\Gamma
\]
and
\[
  \sum_{i\in I} a_iu_i\ne0
  \quad\text{for all nonempty subsets $I\subsetneq\{0,1,\ldots,N\}$.}
\]
Further, the number of such solutions may be bounded solely
in terms of~$N$ and~$\rank(\Gamma)$.
\end{theorem}
\begin{proof}
See~\cite[Theorem~6.2]{MR2582101} or~\cite{MR1923966} for explicit
upper bounds.
\end{proof}

As a warm-up, we first consider the set of elements~$\bfm\in\calM_P$
such that for every~$p\in\calP_{r,n}$, no subsum in the determinant
equation~\eqref{eqn:detAmeq0} equals~$0$.
Theorem~\ref{theorem:Suniteqn} tells us that the equation
\[
  \sum_{\s\in\calS_{r+1}} \sgn(\s) v_\s = 0
\]
has only finitely many solutions in $\PP^{(r+1)!-1}(K)$ with $v_\s\in
R^*$ and such that no subsum equals~$0$. Hence there are only finitely
many possible values for the point
\[
  \bigl[u_{\s,p}(\bfm)\bigr]_{\s\in\calS_{r+1}}\in \PP^{(r+1)!-1}(K).
\]
The assumption that~$P$ is not preperiodic means that at least one
ratio~$\a_t/\a_0$ is not a root of unity. We take~$p\in\calP_{r,n}$
such that~$0$ and~$t$ are in the image of~$p$, which allows us to
apply Lemma~\ref{lemma:MtoPinjective} to conclude that the
map~\eqref{eqn:MtoPn11map} is injective.  Hence there are only
finitely many values for~$\bfm$. Further, the uniformity in
Theorem~\ref{theorem:Suniteqn} gives a uniform upper bound
for~$\calM_P$, and hence for~$\#\Linearspaces^r_{\f,P}$.
\par
We now consider the general case in which one or more subsums
in~\eqref{eqn:detAmeq0} may be equal to~$0$. In this case the
conclusion of Theorem~\ref{theorem:Suniteqn} is false, since we can
scale individual zero subsums. In general, we look at
partitions~$\calI$ of~$\calS_{r+1}$, i.e.,
\[
  \text{$\calI=\{T_1,\ldots,T_s\}$\quad with\quad $T_i\cap T_j=\emptyset$
   \quad and\quad $T_1\cup\cdots\cup T_s=\calS_{r+1}$}.
\]
For each~$p\in\calP_{r,n}$ we want to choose a maximal partition~$\calI_p$
of~$\calS_{r+1}$ such that
\[
  \sum_{\s\in T} \sgn(\s)u_{\s,p}(\bfm)=0
  \quad\text{for all $T\in\calI_p$.}
\]
\par
Let
\[
  \calV = \left\{ [v_\s]_{\scriptscriptstyle\s\in\calS_{r+1}} \in \PP^{(r+1)!-1} : 
      \text{$v_\s\in\G$ and } \sum_{\s\in\calS_{r+1}}\sgn(\s)v_\s=0
     \right\},
\]
and for each partition~$\calI$ of~$\calS_{r+1}$, define
\[
  \calV_\calI
  = \left\{ [v_\s]\in\calV : 
      \begin{array}{l}
        \displaystyle\sum_{\s\in T}\sgn(\s)v_\s=0
        \text{ for all $T\in\calI$, and} \\
        \displaystyle\sum_{\s\in T'}\sgn(\s)v_\s\ne0
        \text{ for all $T'\subsetneq T\in\calI$} \\
      \end{array}
\right\}.
\]
We note that 
\[
  \calV 
  = \bigcup_{\substack{\text{$\calI$ is a partition}\\\text{of $\calS_{n+1}$}\\}}\calV_\calI,
\]
although the $\calV_{\calI}$ are not necessarily disjoint.
\par
For each~$p\in\calP_{r,n}$ we now fix a partition~$\calI_p$
of~$\calS_{n+1}$. This choice of partitions will be fixed for the 
remainder of the proof. We claim that the set
\begin{equation}
  \label{eqn:usmcapVI}
  \left\{ \left(\bigl[u_{\s,p}(\bfm)\bigr]_{\s\in\calS_{r+1}}\right)_{p\in\calP_{r,n}}
         : \bfm\in\calM_P\right\} \cap \prod_{p\in\calP_{r,n}}\calV_{\calI_p}
\end{equation}
is finite (and has order bounded in terms of~$n$). This claim
will complete the proof of Theorem~\ref{theorem:finmanyH}, since from
the definitions it is clear that for all~$p$,
\[
  \left\{ \bigl[u_{\s,p}(\bfm)\bigr]_{\s\in\calS_{r+1}} 
         : \bfm\in\calM_P\right\} \subset \calV,
\]
and Lemma~\ref{lemma:MtoPinjective} tells us that the value of
$\left(\bigl[u_{\s,p}(\bfm)\bigr]_{\s\in\calS_{r+1}}\right)_{p\in\calP_{r,n}}$
determines the value of~$\bfm$.
\par
The definition of~$\calV_\calI$ says that for each~$T\in\calI$, a
certain sum of $S$-units is zero and no subsum is zero. Hence we can
apply Theorem~\ref{theorem:Suniteqn} to each sum over~$T$. In other
words, if we map
\[
  \calV_\calI \longrightarrow \prod_{T\in\calI} \PP^{\#T-1},\qquad
  [v_\s]_{\scriptstyle\s\in\calS_{n+1}} \longmapsto
   \bigl([v_\s]_{\scriptstyle\s\in T}\bigr)_{T\in\calI},
\]
then the image of this map is finite and has order bounded in
terms of~$n$. 

Hence in order to prove that the set~\eqref{eqn:usmcapVI} is finite,
it suffices to show that the map
\begin{equation}
  \label{eqn:usmtoTIPK}
  \begin{aligned}
   \bfF :  \calM_P
    &\longrightarrow \prod_{p\in\calP_{r,n}} \prod_{T\in\calI_p} \PP^{\#T-1},\\
    \bfm
    &\longmapsto \left(\bigl[u_\s(\bfm)\bigr]_{\s\in T}
    \right)_{\substack{p\in\calP_{r,n}\\ T\in\calI_p\\}},
  \end{aligned}
\end{equation}
is injective. 

\begin{lemma}
\label{lemma:subpartition}
For each~$p\in\calP_{r,n}$, let~$\calJ_p$ be a subpartition
of~$\calI_p$, i.e,.  for every~$T\in\calJ_p$ there is a~$T'\in\calI_p$
such that~$T\subset T'$.  Write~$F_\calI$ for the
map~\eqref{eqn:usmtoTIPK} using the~$\calI_p$ partitions, and
write~$F_\calJ$ for the map~\eqref{eqn:usmtoTIPK} using the~$\calJ_p$
partitions. Then
\[
  \bfF_\calI(\bfm) = \bfF_\calI(\widetilde\bfm)
  \quad\Longrightarrow\quad
  \bfF_\calJ(\bfm) = \bfF_\calJ(\widetilde\bfm).
\]
\end{lemma}
\begin{proof}
Let~$p\in\calP_{r,n}$ and
let~$T\in\calJ_p$. Then there is a~$T'\in\calI_p$ with $T\subset T'$. 
The assumption that $\bfF_\calI(\bfm) = \bfF_\calI(\widetilde\bfm)$ means that
here is a~$\l\in K^*$ such that
\begin{equation}
  \label{eqn:usmlustildem}
  u_{\s,p}(\bfm)=\l u_{\s,p}(\widetilde\bfm)
  \quad\text{for all $\s\in T'$.}
\end{equation}
In particular, the equality~\eqref{eqn:usmlustildem} holds for all $\s\in T$,
since $T\subset T'$. Hence
\[
  \bigl[u_\s(\bfm)\bigr]_{\s\in T}=  \bigl[u_\s(\widetilde\bfm)\bigr]_{\s\in T},
\]
and since this is true for all~$T\in\calJ_p$, we conclude that
$\bfF_\calJ(\bfm) = \bfF_\calJ(\widetilde\bfm)$.
\end{proof}

We now resume writing~$\bfF$ for the map~\eqref{eqn:usmtoTIPK}, i.e.,
we drop the~$\calI$ subscript. Suppose that
\[
  \bfF(\bfm)=\bfF(\widetilde\bfm).
\]
This is equivalent to the statement that
\[
  u_{\s,p}(\bfm)u_{\t,p}(\widetilde\bfm) = u_{\t,p}(\bfm)u_{\s,p}(\widetilde\bfm)
  \qquad
  \text{for all $p\in\calP_{r,n}$, all $T\in\calI_p$, and all $\s,\t\in T$.}
\]
Replacing the~$u_{\s,p}$ with their expressions as products of~$\a_i$, this
becomes
\begin{equation}
  \label{eqn:prodaiksietc}
  \prod_{i=0}^n \a_{p(i)}^{k_{\s(i)}+\tilde k_{\t(i)}-k_{\t(i)}-\tilde k_{\s(i)}} = 1
  \quad
  \text{for all $p\in\calP_{r,n}$, all $T\in\calI_p$, and all $\s,\t\in T$.}
\end{equation}

We now invoke the assumption that the relation set for the point~$P$
is trivial, i.e.,
\[
  \framebox{$\Relations(P) = \{\bfzero\}$.}
\]
We note that
\[
  \sum_{i=0}^n \bigl(k_{\s(i)}-k_{\t(i)}\bigr)
  =   \sum_{i=0}^n \bigl(\tilde k_{\t(i)}-\tilde k_{\s(i)}\bigr) = 0,
\]
so the assumption $\Relations(P) = \{\bfzero\}$ implies that
the exponents in~\eqref{eqn:prodaiksietc} all vanish. Hence
\[
  k_{\s(i)}-k_{\t(i)} = \tilde k_{\t(i)}-\tilde k_{\s(i)}
  \quad
  \text{for all $p\in\calP_{r,n}$, all $T\in\calI_p$, 
          all $\s,\g\in T$, and all $0\le i\le n$.}
\]
Rewriting using $k_i=k_i(\bfm)=d^{m_i}$ yields
\begin{multline}
  \label{eqn:dmsidmtieqtilde}
  d^{m_{\s(i)}} - d^{m_{\t(i)}} =   d^{\tilde m_{\s(i)}} - d^{\widetilde m_{\t(i)}}\\
  \text{for all $p\in\calP_{r,n}$, all $T\in\calI_p$, 
          all $\s,\g\in T$, and all $0\le i\le n$.}
\end{multline}
Since the entries of~$\bfm$ are distinct, as are the entries
of~$\widetilde\bfm$, we can use Lemma~\ref{lemma:dadbeqdxdy} to deduce
that
\begin{equation}
  \label{eqn:sinetiimpliesmseqx}
  \s(i)\ne\t(i)
  \qquad\Longrightarrow\qquad
  m_{\s(i)}=\widetilde m_{\s(i)}
  \quad\text{and}\quad
  m_{\t(i)}=\widetilde m_{\t(i)}.
\end{equation}
This holds for all $0\le i\le n$, all $p\in\calP_{r,n}$, all
$T\in\calI_p$, and all $\s,\t\in T$.  It will be more convenient to
use the contrapositive of~\eqref{eqn:sinetiimpliesmseqx}. So suppose
that~$m_t\ne\widetilde m_t$ for some~$t$. Taking~$i=\s^{-1}(t)$ for
some~$\s\in T$, we deduce that
\[
  m_t\ne\widetilde m_t
  \quad\Longrightarrow\quad
  \s\bigl(\s^{-1}(t)\bigr)=\t\bigl(\s^{-1}(t)\bigr).
\]
The conclusion may be rewritten as $\t^{-1}(t)=\s^{-1}(t)$, so we have
proven the following key implication:
\begin{equation}
  \label{eqn:sinetiimpliesmseq}
  m_t\ne\widetilde m_t
  \quad\Longrightarrow\quad
  \forall p\in\calP_{r,n},\;
  \forall T\in\calI_p,\;
  \forall \s,\t\in T,\quad
  \t^{-1}(t)=\s^{-1}(t).
\end{equation}
\par
We would like to show that~\eqref{eqn:sinetiimpliesmseq} gives enough
relations to force~$\bfm=\widetilde\bfm$, but we will need to use
the super-spanning assumption, i.e., the assumption that~$A_\bfm$
and~$A_{\widetilde\bfm}$ have super-rank~$r$, to eliminate some
exceptional cases for which there are not enough relations.

\begin{definition}
For each $0\le t\le r$ and each $0\le j\le r$, we define sets
\[
  T_{j}^t = \{ \s\in\calS_{r+1} : \s(j) = t \}.
\]
For each~$t$, this gives a partition
\[
  \calI^t_{\bullet} = \{T_{0}^t,T_{1}^t,\ldots,T_{r}^t\}
\]
of~$\calS_{r+1}$. We will say that a partition of~$\calS_{r+1}$ is
\emph{exceptional} if it is a subpartition of~$\calI^t_{\bullet}$ for
some $0\le t\le r$.  In particular, each partition~$\calI^t_{\bullet}$
is itself exceptional.
\end{definition}

\begin{lemma}
\label{lemma:nonexceptinject}
Suppose that for some~$p\in\calP_{r,n}$, the partition~$\calI_p$ is
not exceptional.  Then the map~$\bfF$ defined by~\eqref{eqn:usmtoTIPK}
is injective.
\par
More precisely, if $\bfF(\bfm)=\bfF(\widetilde\bfm)$ with
$m_t\ne\widetilde m_t$, then every~$\calI_p$ is a subpartition
of~$\calI_{\bullet}^t$.
\end{lemma}
\begin{proof}
We prove the second statement, so we assume that
\[
  \bfF(\bfm)=\bfF(\widetilde\bfm)
  \quad\text{and}\quad
  m_t\ne\widetilde m_t,
\]
and we will prove that~$\calI_p$ is a subpartition
of~$\calI^t_{\bullet}=\{T_{0}^t,\ldots,T_{n}^t\}$.
\par
Let $T\in\calI_p$ be any set in the partition, and let~$\s,\t\in T$.  
Applying~\eqref{eqn:sinetiimpliesmseq}, we conclude that
\[
  \t^{-1}(t)=\s^{-1}(t).
\]
Hence the set
\[
  \bigl\{\s^{-1}(t) : \s\in T\bigr\}
\]
contains only one number, which we denote by~$j(T)$. In other words,
\[
  \s\bigl(j(T)\bigr) = t \quad\text{for all $\s\in T$,}
\]
so by definition of~$T_j^t$, this means that~$T\subset T_{j(T)}^t$,
where we stress that the index~$t$ does not depend on~$p$ or~$T$. Thus
every~$T\in\calI_p$ is contained in one of the sets in the
partition~$\{T_{0}^t,\ldots,T_{n}^t\}$,
so~$\calI_p\subset\calI_{\bullet}^t$.  In particular,~$\calI_p$ is an
exceptional partition, which completes the proof of
Lemma~\ref{lemma:nonexceptinject}.
\end{proof}

\begin{remark}
If any~$\calI_p$ is not exceptional, then
Lemma~\ref{lemma:nonexceptinject} says that~$\bfF$ is injective.  We
now indicate how exceptional partitions~$\calI_p$ can lead to~$\bfF$
being non-injective.  From Lemma~\ref{lemma:subpartition}, it suffices
to look at the case that~$\calI_p\subset\calI_{\bullet}^t$ for all~$p$
and some~$t$, since if these lead to non-injective maps~$\bfF$, then
so do their subpartitions.
\par
For~$\s\in T_{j}^t$, we compute 
\begin{equation}
  \label{eqn:usmajktprod}
  u_{\s,p}(\bfm)
  = \prod_{i=0}^r \a_{p(i)}^{k_{\s(i)}}
  = \a_{p(j)}^{k_t} \prod_{\substack{i=0\\i\ne j\\}}^r \a_{p(i)}^{k_{\s(i)}}
  = \a_{p(j)}^{k_t} \prod_{\substack{i=0\\i\ne t\\}}^r \a_{p(\s^{-1}(i))}^{k_i},
\end{equation}
where the last equality uses that fact that $\s(j)=t$ for all $\s\in
T_j^t$.  The quantity $\a_{p(j)}^{k_t}$ does not depend on~$\s$, so it
may be canceled from homogeneous coordinates, yielding
(where we recall that $k_i=d^{m_i}$)
\begin{equation}
  \label{eqn:usmsTjtdndt}
  \bigl[u_{p,\s}(\bfm)\bigr]_{\s\in T_j^t} 
  = \left[\prod_{\substack{i=0\\i\ne t\\}}^n \a_{p(\s^{-1}(i))}^{d^{m_i}}\right]_{\s\in T_j^t}.
\end{equation}
Note that this formula for $\bigl[u_{p,\s}(\bfm)\bigr]_{\s\in T_j^t}$
does not involve~$m_t$.  Hence $\bfF(\bfm)$ does not depend on the
$t^{\text{th}}$-coordinate of~$\bfm$.  Thus~$\bfF$ is not injective if
we consider it to be a map on the set of all integer vectors
$\bfm=(m_0,\ldots,m_r)$. We will need to use the assumption
that~$\bfm\in\calM_P$, i.e., that~$A_\bfm$ has super-rank equal
to~$r$, to rule out with this potential non-injectivity.
\end{remark}

We resume the proof of Theorem~\ref{theorem:finmanyH}, so in
particular assuming that~$\calR(P)=\{\bfzero\}$. Recall that for each
$p\in\calP_{r,n}$ we have fixed a partition~$\calI$ of~$\calS_{r+1}$
and used it in~\eqref{eqn:usmtoTIPK} to define a map~$\bfF$.  As
indicated earlier, it suffices to prove that the
set~\eqref{eqn:usmcapVI} is finite, and for this it suffices to prove
that the map~$\bfF$ is injective. Lemma~\ref{lemma:nonexceptinject}
says that~$\bfF$ is injective unless every~$\calI_p$ is exceptional
for the same index~$t$, i.e., unless there is an index~$t$ such
every~$\calI_p$ is contained in~$\calI_{\bullet}^t$. We deal with
these exceptional cases in the following lemma, which says that in
these cases, the set~\eqref{eqn:usmcapVI} is empty.

\begin{lemma}
\label{lemma:exceptsprrnk}
Suppose that there is an index~$t$ such that $\calI_p\subset\calI_{\bullet}^t$
for every $p\in\calP_{r,n}$. Let
$\bfm=(m_0,\ldots,m_n)\in\ZZ^{n+1}$ with $0\le m_0<m_1<\cdots<m_n$
have the property that 
\begin{equation}
  \label{eqn:usminVI}
  \bigl(u_{\s,p}(\bfm)\bigr)_{\s\in\calS_{r+1}} \in \calV_{\calI_p}
  \qquad\text{for all $p\in\calP_{r,n}$.}
\end{equation}
Then $\bfm\notin\calM_P$, i.e., the matrix $A_\bfm$ does not have
super-rank equal to~$r$. More precisely, if we delete the~$t^{\text{th}}$~row
of~$A_\bfm$, the resulting matrix has rank~\text{$r-1$}.
\end{lemma}
\begin{proof}  
Assumption~\eqref{eqn:usminVI} and the definition of~$\calV_{\calI_p}$
say that
\begin{equation}
  \label{eqn:sumsTusm0}
  \sum_{\s\in T} \sgn(\s) u_{\s,p}(\bfm)=0
  \quad\text{for all $T\in\calI_p$.}
\end{equation}
We are assuming that~$\calI_p\subset\calI_\bullet^t$, so
each~$T'\in\calI_\bullet^t$ is a union of elements of~$\calI_p$.
Summing~\eqref{eqn:sumsTusm0} over the~$T$ whose union is~$T'$, we
find that~\eqref{eqn:sumsTusm0} is true for the
partition~$\calI_\bullet^t$. Thus it suffices to prove the
lemma under the assumption that~$\calI_p$ is equal to the maximal
exceptional partition~$\calI_\bullet^t$.
\par
As computed earlier, see~\eqref{eqn:usmajktprod}, we have
\[
  u_\s(\bfm) = \a_{p(j)}^{k_t}\prod_{i\ne j} \a_{p(i)}^{k_{\s(i)}}
  \qquad\text{for all $\s\in T_j^t$.}
\]
Hence in the sum~\eqref{eqn:sumsTusm0} with $T=T_j^t$, we can
cancel~$\a_{p(j)}^{k_t}$ from every term, which yields
\[
  \sum_{\s\in T_j^t} \sgn(\s) \prod_{i\ne j} \a_{p(i)}^{k_{\s(i)}} = 0.
\]
We observe that this sum is exactly the determinant of the matrix
obtained by deleting the $t^{\text{th}}$~row and the
$p(j)^{\text{th}}$~column from the matrix~$A_{\bfm,p}$. (This is because the
sum consists of the terms for which~$i$ is never equal to~$p(j)$
and~$\s(i)$ is never equal to~$t$.)  
\par
The value of~$t$ is fixed, but~$p$ and~$j$ are arbitrary, so we have
proven that if we delete the $t^{\text{th}}$~row of~$A_\bfm$, then the
resulting \text{$r$-by-$(n+1)$} dimensional matrix has rank at
most~$r-1$, since all of its \text{$r$-by-$r$} minors vanish. By
the definition of super-rank, it follows that~$A_\bfm$ does not have
super-rank~$r$, which completes the proof of
Lemma~\ref{lemma:exceptsprrnk}
\end{proof}  

We now summarize how the preceding pieces fit together to prove
Theorem~\ref{theorem:finmanyH}.

\begin{enumerate}
\setlength{\itemsep}{0pt}
\item
Our goal is to prove that the set of exceptional $(r-1)$-dimen\-sional
linear subspaces~$\calL_{\f,P}^r$ is finite.
\item
Lemma~\ref{lemma:MtoHsurjective} says that the set $\calM_P$ of
$(r+1)$-tuples such that~$A_\bfm$ has super-rank~$r$ maps
onto~$\calL_{\f,P}^r$, so it suffices to prove that~$\calM_P$ is
finite.
\item
For each~$p\in\calP_{r,n}$, i.e., for each choice of an
$(r+1)$-by-$(r+1)$ minor of the matrix~$A_\bfm$, we fix a
partition~$\calI_p$ of~$\calS_{r+1}$ that describes the minimal
subsums of the determinant~$\det A_{\bfm,p}$ that vanish.
\item
The terms in the expansion of~$\det A_{\bfm,p}$ lie in a finitely
generated subgroup of~$K^*$, so Theorem~\ref{theorem:Suniteqn} says
that there are only finitely many possibilities for the terms in each
subsum, where the terms are viewed as a point in projective space.
\item
This implies that the image of the map~$\bfF$ defined
by~\eqref{eqn:usmtoTIPK} is finite, so it suffices to show that~$\bfF$
is injective for all choices of the partitions~$\calI_p$.
\item
Lemma~\ref{lemma:nonexceptinject} says that the map~$\bfF$ is
injective unless there is an index~$t$ such that every~$\calI_p$ is a
subpartition of the exceptional partition~$\calI_{\bullet}^t$.  (A key
point here is that one~$t$ works for every~$p$.) 
\item
It remains to deal with the case that there is an index~$t$ such that
every~$\calI_p$ is a subpartition of~$\calI_{\bullet}^t$. But in this
case, Lemma~\ref{lemma:exceptsprrnk} says that the associated
$(r+1)$-tuples~$\bfm$ give matrices~$A_\bfm$ that have rank~$r-1$ when
their~$t^{\text{th}}$~rows are deleted. Thus~$A_\bfm$ does not have
super-rank~$r$, so~$\bfm\notin\calM_P$.
\end{enumerate}

This completes the proof of Theorem~\ref{theorem:finmanyH}.
\end{proof}

\section{The size of the intersection $\OO_\f(P)\cap L$}  
\label{section:sizeintersection}

In this section we give the proof of Corollary~$\ref{cor:maincor}$,
which we restate for the convenience of the reader

\begin{cor}
Let $\f$ and $P$ be as in the statement of
Theorem~$\ref{theorem:mainthm}$.  Then for any linear subspace $L$,
the intersection \text{$\OO_{\f}(P)\cap L$} is finite, and its size is
bounded solely in terms of~$n$, independent of $P$ and $d$.
\end{cor}
\begin{proof}
Let us return to the map $\calM_P\longrightarrow \calL_{\f,P}^r$
defined in Lemma~\ref{lemma:MtoHsurjective}.  Fix a linear space $L$
in the codomain, so 
\[
  \#\left(L\cap\OO_\f(P)\right)\geq r + 1.
\]
Then the preimage of $L$ under this map, i.e.,
\[
  \calM_P(L) := \left\{\bfm \in\calM_P : L_{\bfm} = L\right\},
\]
consists of strictly increasing $(r + 1)$-tuples $\bfm$ such that
\[
  \bigl\{\phi^{m_0}(P),\ldots,\phi^{m_r}(P) \bigr\} \subseteq L\cap\OO_\f(P).
\]
Thus a bound on $\#\calM_P$ solely in term of $n$ gives an analogous
bound on \text{$\#\left(L\cap\OO_\f(P)\right)$} for any linear space
$L$. Using this fact, Corollary~\ref{cor:maincor} follows from the proof of
Theorem~\ref{theorem:finmanyH}.
\end{proof}

\section{Orbits that are not Zariski dense} 

It is interesting to ask if we can weaken the hypotheses of
Conjecture~\ref{conjecture:finmanylinsp}.  This is already a
nontrivial question when $\phi$ is the $d^{\textup{th}}$-power map, in
which case we ask what happens if we allow non-trivial multiplicative
relations among the coordinates of of~$P$.  We now show that the
uniformity part of conjecture fails, i.e., for fixed~$d$ and~$n$, the
number of exceptional subspaces may grow as the point~$P$ varies.
Although this example is somewhat artificial, it shows that some
condition on~$P$ or its orbit is needed if one is to drop the
assumption in Conjecture~\ref{conjecture:finmanylinsp}
that~$\calO_\f(P)$ be Zariski dense in~$\PP^n$.

\begin{proposition}
\label{proposition:manyexcepthyps}
Let $\phi_d\colon\PP^n\to\PP^n$ be the $d^{\textup{th}}$-power
map~\eqref{eqn:dthpowermap}, let~$\ell\ge3$ be a prime such that~$d$ is a
primitive root modulo~$\ell$, let~$\zeta_\ell$ denote a
primitive~$\ell^{\text{th}}$-root of unity, let
\[
  P = [1,\zeta_\ell,\a_2,\ldots,\a_n]
\]
be a point with~$\a_2,\ldots,\a_n$ multiplicatively independent, and
let~$V$ be the \textup(reducible\textup) hypersurface
\[
  V = \left\{\bfx\in\PP^n :  \frac{x_1^{\ell} - x_0^{\ell}}{x_1-x_0}=0 \right\}.
\]
Then for every~$0<i<\ell$, the hyperplane
\[
  H_i = \{\bfx\in\PP^n : x_1=\zeta_\ell^ix_0\} \subset \PP^n
\]
is an exceptional hyperplane for the map~$\f_d$.
\end{proposition}
\begin{proof}
It is clear that~$V$ is the union of the~$H_i$ for $0<i<\ell$.  We
also observe that
\[
  \f^n(P) = [1,\zeta_\ell^{d^n},\a^{d^n}_2,\ldots,\a^{d^n}_n] \in H_i
  \quad\Longleftrightarrow\quad
  d^n\equiv i\pmod{\ell}.
\]
Since~$d$ is a primitive root modulo~$\ell$, there is a unique
integer $0<n_i<\ell$ such that
\[
  d^n\equiv i\pmod{\ell}
  \quad\Longleftrightarrow\quad
  n \equiv n_i \pmod{\ell-1}.
\]
Thus each~$H_i$ contains infinitely many points of~$\Orbit_{\f_d}(P)$,
and the multiplicative independence of~$\a_2,\ldots,\a_n$ implies that
$\Orbit_{\f_d}(P)\cap H_i$ is Zariski dense in~$H_i$. Therefore~$H_i$
is an exceptional subspace for~$\f_d$ and~$P$.
\end{proof}

\begin{corollary}
If we drop the assumption in Theorem~$\ref{theorem:finmanyH}$ that
$\Relations(P)=0$, then there does not exist a bound for
$\#\calL_{\f,P}^n$ that depends only on~$n$ and~$d$, independent of
the point~$P$.
\end{corollary}
\begin{proof}
Proposition~\ref{proposition:manyexcepthyps} says that for every
prime~$\ell$ such that~$d$ is a primitive root modulo~$\ell$, we can
find a point~$P_\ell$ such that
\[
  \#\calL_{\f_d,P_\ell}^n\ge\ell-1.
\]
To prove the corollary, it suffices to note that there exist many
integers~$d$ with the property that they are primitive roots for
infinitely many primes~$\ell$. See for
example~\cite{MR762358,MR830627}, where it is proven that such~$d$
exist, and indeed are quite common. (Of course, Artin's conjecture
says that aside from the obvious exceptions, every~$d$ has this
property.)
\end{proof}

On the other hand, we are able to show by a detailed case-by-case
analysis that the conjecture holds for some choices of $P$ for which
the relation set is non-trivial, i.e., for which the orbit $\OO_\f(P)$
is \emph{not} Zariski dense. The following example demonstrates how
such results are proven, while also illustrating the case-by-case
analysis that makes it difficult to prove a general theorem.

\begin{prop}
\label{proposition:nonzerorel}
Let $\phi\colon\PP^3\to\PP^3$ be the $d^{\textup{th}}$-power
map~\eqref{eqn:dthpowermap}, and let
\[
  P = [\a_0,\a_1,\a_2,\a_3]\in\PP^3
  \quad\text{satisfy}\quad
  \a_0\a_1=\a_2\a_3,
\]
so~$P$ lies on the quadric surface
\[
  V = \bigl\{ [x_0,x_1,x_2,x_3] \in \PP^3 :  x_0x_1 = x_2x_3 \bigr\}.
\]
Assume further that~$\OO_\f(P)$ is Zariski dense in~$V$.  Then for all
$r\ge1$, the set of exceptional linear spaces~$\Linearspaces_{\f,P}^r$
is finite, and $\#\Linearspaces_{\f,P}^r$ may be bounded solely in
terms of~$n$ and~$d$.
\end{prop}

\begin{proof}[Proof Sketch]
We note that in our notation, the assumption that
$\overline{\OO_\f(P)}=V$ is equivalent to the assumption that the
relation set $\Relations(P)$ has rank~$1$ and is generated by
$(1,1,-1,-1)$.  Since $P$ is not preperiodic and does not lie on any
coordinate hyperplane, much of the proof of
Theorem~\ref{theorem:finmanyH} carries over with no change.  To
complete the proof of Proposition~\ref{proposition:nonzerorel}, it
remains to show that for any partition $\calI_p$ that is not
exceptional, the map
\begin{equation}
  \bfF :  \calM_P
    \longrightarrow \prod_{p\in\calP_{r,n}} \prod_{T\in\calI_p} \PP^{\#T-1},
  \qquad
  \bfm
    \longmapsto \left(\bigl[u_\s(\bfm)\bigr]_{\s\in T}
        \right)_{\substack{p\in\calP_{r,n}\\ T\in\calI_p\\}},
\end{equation}
is injective.
\par
Let $\bfm, \widetilde{\bfm} \in\calM_P$ be such that $\bfF(\bfm) =
\bfF(\widetilde{\bfm})$.  This implies that 
for all $p\in\calP_{r,n}$, all $T\in\calI_p$, and all $\s,\t\in T$,
the vector
\begin{equation}
  \label{eq:relnimplication}
  v_{\s,\t} = \left({k_{\s(p^{-1}(i))}+\tilde k_{\t(p^{-1}(i))}-
    k_{\t(p^{-1}(i))}-\tilde k_{\s(p^{-1}(i))}}\right)_{0\le i\le 3}
  \quad\text{is in}\quad \Relations(P).
\end{equation}
(If $i$ is not in the image of $p$, we set the
$i^{\textup{th}}$-coordinate of~$v_{\s,\t}$ to be~$0$.)  Since
\[
  \Relations(P)\cap \bigl\{(e_0,e_1,e_2,e_3) : e_i=0\bigr\}
   = \{\bfzero\}
\]
for all $i$, if $r<n$ we can use the same argument as in the proof of
Theorem~\ref{theorem:finmanyH}.
\par
We now restrict to the case that $r = n$, so $\calP_{r,n} =
\{\textup{id}\}$.  First we assume that there exists a $T \in \calI$
and a $\s, \t \in T$ such that $\s^{-1}\t$ has no fixed points.  A
case-by-case analysis of the possibilities for $\s^{-1}\t$ shows
that~\eqref{eq:relnimplication} forces that $\bfm = \widetilde{\bfm}$.
We illustrate with two cases. 
\par
Suppose that~$\t^{-1}\s=(0123)$. Since~$\Relations(P)$ is generated
by~$(1,1,-1,-1)$, the fact that~$v_{\s,\t}$ is in~$\Relations(P)$ implies
that the second and third coordinates of~$v_{\s,\t}$ are negatives of
one another, i.e., 
\begin{equation}
  \label{eqn:ks1tkt1}
  k_{\s(1)}+\tilde k_{\t(1)} - \tilde k_{\s(1)}-k_{\t(1)}
  =
  -k_{\s(2)}-\tilde k_{\t(2)} + \tilde k_{\s(2)}+k_{\t(2)}.
\end{equation}
The assumption that $\t^{-1}\s=(0123)$ implies that $\s(1)=\t(2)$, which
allows us to simplify~\eqref{eqn:ks1tkt1} to 
\[
  \tilde k_{\t(1)} - k_{\t(1)}  =   \tilde k_{\s(2)} - k_{\s(2)}.
\]
Using~$k_i=d^{m_i}$ and Lemma~\ref{lemma:dadbeqdxdy} as usual, we
conclude that $k_{\s(2)}=\tilde k_{\s(2)}$ and $k_{\t(1)}=\tilde
k_{\t(1)}$. Substituting this into the relation equation and using the
fact that the first two coordinates of~$v_{\s,\t}$ are the same, we
find that $k_{\s(3)}=\tilde k_{\s(3)}$ and $k_{\s(1)}=\tilde
k_{\s(1)}$. This argument works, \emph{mutatis mutandis},
if $\t^{-1}\s$ is any four cycle, as well as when~$\t^{-1}\s$ is
either $(02)(13)$ or $(03)(12)$.
\par
Next suppose that $\t^{-1}\s=(01)(23)$. This choice of~$\t^{-1}\s$
means that
\[
  \s(0)=\t(1),\quad \s(1)=\t(0),\quad \s(2)=\t(3),\quad \s(3)=\t(2).
\]
Substituting these into the definition~\eqref{eqn:ks1tkt1}
of~$v_{\s,\t}$, we see that the vector~$v_{\s,\t}$ has the form
$(X,-X,Y,-Y)$. But~$\Relations(P)$ is generated by~$(1,1,-1,-1)$,
so~$v_{\s,\t}=\bfzero$, which implies as usual
that~$\bfk=\tilde{\bfk}$.
\par
We are left to consider the case that for all $T \in \calI$ and for
all $\s,\t \in T$, there is an $i$, depending on $\s, \t$, such that
$\s(i) = \t(i)$.  This implies that the $i^{\textup{th}}$-component of
$v_{\sigma,\tau}$ is $0$, and so $v_{\sigma,\tau}$ must be the zero
vector.  Hence if there is a $t$ such that $m_t \neq \widetilde{m}_t$,
then $\tau^{-1}(t) = \sigma^{-1}(t)$ for all $T\in \calI$ and for all
$\s,\t \in T$.  This says that $\calI$ is exceptional, which completes
the proof.
\end{proof}

\begin{remark}
The proof of Proposition~\ref{proposition:nonzerorel} gives a more
general result, namely that there is a uniform bound
for~$\#\calL^r_{\f,P}$ provided that for every choice~$H_1,\ldots,H_{n-r}$ 
of~$n-r$ coordinate hyperplanes, the relation set satisfies
\[
  \Relations(P)\cap H_1\cap\cdots\cap H_{n-r} = \{\bfzero\}.
\]
\end{remark}



\end{document}